\documentclass[a4paper,12pt,reqno]{amsart}
\usepackage{amsfonts}
\usepackage{amsmath}
\usepackage{amssymb}
\usepackage{mathrsfs}
\usepackage{hyperref}
\numberwithin{equation}{section}

\usepackage{graphicx}
 \newtheorem{thm}{Theorem}[section]
 \newtheorem{cor}[thm]{Corollary}
 \newtheorem{lem}[thm]{Lemma}
 
 \theoremstyle{definition}
 \newtheorem{defn}[thm]{Definition}
 \theoremstyle{remark}
 \newtheorem{rem}[thm]{Remark}
 
 \numberwithin{equation}{section}

\newcommand{\bba}{\mathcal{B}}
\DeclareMathOperator{\Tr}{Tr}

\newcommand{\er}{\mathbb{R}}

\newcommand{\ern}{{\mathbb{R}}^n}

\newcommand{\vel}{L^{p(\cdot)}}
\newcommand{\velp}{L^{p'(\cdot)}}
\newcommand{\velq}{L^{q(\cdot)}}

\newcommand{\bi}{\begin{itemize}}
\newcommand{\ei}{\end{itemize}}
\newcommand{\be}{\begin{enumerate}}
\newcommand{\ee}{\end{enumerate}}
\newcommand{\beq}{\begin{equation}}
\newcommand{\eq}{\end{equation}}

\newcommand{\efet}{{\mathcal{F}}_{\mathbb{T}^n}}

\newcommand{\tn}{\mathbb{T}^n}

\newcommand{\zn}{\mathbb{Z}^n}

\def\Rn{{{\mathbb R}^n}}
\def\Tn{{{\mathbb T}^n}}
\def\Zn{{{\mathbb Z}^n}}

\def\SU2{{{\rm SU(2)}}}
\def\SO3{{{\rm SO(3)}}}
\def\lapsu2{{{\mathcal L}_\SU2}}

\begin{document}

%
%
%
%
%
%
%
%
%
\title
[THE BOUNDED APPROXIMATION PROPERTY OF $L^{p(\cdot)}$ AND NUCLEARITY]
 {THE BOUNDED APPROXIMATION PROPERTY OF VARIABLE LEBESGUE SPACES AND NUCLEARITY}

\author[Julio Delgado]{Julio Delgado}


\address{Department of Mathematics\\
Imperial College London\\
180 Queen's Gate, London SW7 2AZ\\
United Kingdom}

\email{j.delgado@imperial.ac.uk}

\thanks{The first author was supported by the
Leverhulme Research Grant RPG-2014-02.
The second author was supported by
 the EPSRC Grant EP/K039407/1. No new data was collected or generated during the coiurse of research.}

\author{Michael Ruzhansky}

\address{Department of Mathematics\\
Imperial College London\\
180 Queen's Gate, London SW7 2AZ\\
United Kingdom}

\email{m.ruzhansky@imperial.ac.uk}


\date{\today}
\begin{abstract}
In this paper we prove the bounded approximation property for variable exponent Lebesgue spaces, study the concept of 
 nuclearity on such spaces and apply it to trace formulae such as the 
Grothendieck-Lidskii formula. We apply the obtained results to derive criteria for nuclearity and trace formulae
for periodic operators on $\Rn$ in terms of global symbols.
\end{abstract}

\maketitle
\section{Introduction}
The approximation property on a Banach space arises in the study of the concept of trace and was first introduced in its current shape by Grothendieck in his monumental work  \cite{gro:me}. A particular importance for a Banach space enjoying this  property  is that the trace can be defined and consequently the Fredholm's determinant leading to numerous further developments. Indeed, this problematic finds itself closely related to a wide range of analysis areas:  operator theory, spectral analysis, harmonic analysis, functional analysis, PDEs.

In \cite{ap:enflo} Enflo constructed a counterexample to the approximation property in Banach spaces. A more natural counterexample was then found by Szankowski \cite{ap:sz} who proved that $B(H)$ does not have the approximation property. More recently these properties have been intensively investigated  by Figiel, Johnson, Pelczy\'nski and Szankowski in  
\cite{Figiel-et-al:IJM}, \cite{jsza:happy}. Alberti, Cs\"ornyei, Pelczy\'nski and Preiss 
\cite{Alberti-Csornyei-Pelczynski-Preiss:BV} 
established the bounded approximation property (BAP) for functions of bounded variations, and
Roginskaya and Wojciechowski \cite{Roginstaya:arxiv-2014}
for Sobolev spaces $W^{1,1}.$  The authors have recently established the metric approximation property for mixed-norm $L^{p}$,
 modulation and Wiener amalgam spaces in \cite{drwap2:app}, see also \cite{DRW-J-spectr-theory}. Other works on the bounded approximation property can be found in \cite{lo:bapp1}, \cite{lo:bapp2}. 
 A weak approximation property has been introduced and investigated in \cite{lo:app}.   The fact that the approximation property does not imply the bounded approximation property was proved in \cite{fj:app}. 
 For a historical perspective and an introduction to the subject the reader can be referred to Pietsch's book \cite[Section 5.7.4]{Pietsch:bk-history-2007} and the recent revisited presentation on the Grothendieck's classical work  by Diestel, Fourie and Swart \cite{diestel:gro}. The monograph \cite{ray:tp} contains a more accessible introduction to the topic as well as several examples of spaces enjoying approximation properties. An introductory survey to the concept of trace on Banach spaces appeared in \cite{dr:gro} by Robert. 

To formulate the notions more precisely, let $\mathcal{B}_1, \mathcal{B}_2$ be Banach spaces.
A linear operator $T$ from $\bba_1$ to $\bba_2$ is called {\em nuclear} if there exist sequences
$(x_{n}^\prime)\mbox{ in } \bba_1 '$ and $(y_n) \mbox{ in } \bba_2$ such that
$$
Tx= \sum\limits_{n=1}^{\infty} \left <x,x_{n}'\right>y_n \,\mbox{ and }\,
\sum\limits_{n=1}^{\infty} \|x_{n}'\|_{\bba_1 '}\|y_n\|_{\bba_2} < \infty.
$$
This definition agrees with the concept of a trace class operator in
the setting of Hilbert spaces. The set of nuclear operators from $\bba_1$ into $\bba_2$ forms the ideal of nuclear operators $\mathcal{N}(\bba_1, \bba_2)$ endowed with the norm
\beq\nonumber 
N(T)=\inf\{\sum\limits_{n=1}^{\infty} \|x_{n}'\|_{\bba_1 '}\|y_n\|_{\bba_2} : 
T=\sum\limits_{n=1}^{\infty} x_{n}'\otimes y_n \}.\eq
If $\bba=\bba_1=\bba_2$, it is natural to attempt to define the trace of $T\in\mathcal{N}(\bba)$ by
\begin{equation}\label{EQ:Trace}
\Tr (T):=\sum\limits_{n=1}^{\infty}x_{n}'(y_n),
\end{equation}
where $T=\sum\limits_{n=1}^{\infty}x_{n}'\otimes y_n$ is a
representation of $T$. Grothendieck \cite{gro:me} proved that the trace $\Tr(T)$ is well defined for 
 all nuclear operators $T\in\mathcal{N}(\bba)$ if and only if the Banach space
 $\bba$ has the {\em approximation property}
  (see also Pietsch \cite{piet:book} or Defant and Floret \cite{df:tensor}),  which means that for every compact
   set $K$ in $\bba$ and for every $\epsilon >0$ there exists $F\in \mathcal{F}(\bba) $ such that
\[\|x-Fx\|<\epsilon\quad \textrm{ for all } x\in K,\]
where we have denoted by $\mathcal{F}(\bba)$ the space of all finite rank bounded linear operators
 on $\bba$. We denote by $\mathcal{L}(\bba)$ the Banach algebra of bounded linear operators on $\bba$. 
 
There are more related approximation properties, e.g. if in the definition above
 the operator $F$ satisfies $\|F\|\leq M$ for a fixed $M>0$ one says that $\bba$ possesses the 
 {\em bounded approximation property}. In the case $M=1$ one says that $\bba$ has the 
 {\em metric approximation property}. 
 The fact that the classical spaces $C(X)$, where $X$ is a compact topological space and $L^p(\mu)$ for $1\leq p<\infty$  satisfy the metric approximation property can be found in \cite{piet:book1}.

As we know from Lidskii \cite{li:formula},
in Hilbert spaces the operator trace is equal to the
sum of the eigenvalues of the operator counted with multiplicities. This property is nowadays called the
Lidskii formula.
An important feature on Banach  spaces even endowed with the approximation property is that the Lidskii formula does not hold in general for nuclear operators. 
Thus, in the setting of Banach spaces,
Grothendieck \cite{gro:me}
introduced a more restricted class of operators where Lidskii formula holds, 
this fact motivating the following definition.

Let $\bba_1, \bba_2$ be Banach spaces and let $0<r\leq 1$. A linear operator $T$
from $\bba_1$ into $\bba_2$ is called {\em r-nuclear} if there exist sequences
$(x_{n}^{\prime})\mbox{ in } \bba_1' $ and $(y_n) \mbox{ in } \bba_2$ so that
\beq 
Tx= \sum\limits_{n=1}^{\infty} \left <x,x_{n}'\right>y_n \,\mbox{ and }\,
\sum\limits_{n=1}^{\infty} \|x_{n}'\|^{r}_{\bba_1'}\|y_n\|^{r}_{\bba_2} < \infty.\label{rn}
\eq
We associate a quasi-norm $n_r(T)$ by
\[
n_r(T)^r:=\inf\{\sum\limits_{n=1}^{\infty} \|x_{n}'\|^{r}_{\bba_1'}\|y_n\|^{r}_{\bba_2}\},
\]
where the infimum is taken over the  representations of $T$ as in \eqref{rn}. 
When $r=1$ the $1$-nuclear operators agree with 
the nuclear operators, and as already mentioned,
in that case this definition also agrees with the concept of trace class operators
 in the setting of Hilbert spaces ($\bba_1=\bba_2=H$). More generally, Oloff proved in \cite{Oloff:pnorm} that the class of $r$-nuclear
operators coincides with the Schatten class $S_{r}(H)$ when $\bba_1=\bba_2=H$ is a Hilbert space and 
$0<r\leq 1$. Moreover, Oloff proved that 
\beq\label{olo1}\|T\|_{S_r}=n_r(T),\eq
where $\|\cdot\|_{S_r}$ denotes the classical Schatten quasi-norms in terms of singular values.

In \cite{gro:me} Grothendieck proved that if $T$ is $\frac 23$-nuclear from $\bba$ into $\bba$ for a Banach space $\bba$, then
\beq\Tr(T)=\sum\limits_{j=1}^{\infty}\lambda_j,\label{lia1}\eq
where $\lambda_j\,\, (j=1,2,\dots)$ are the eigenvalues of $T$ with multiplicities taken into account,
and $\Tr(T)$ is as in \eqref{EQ:Trace}.
 Grothendieck also established its applications to the distribution of eigenvalues of operators
in Banach spaces. We refer to \cite{dr13a:nuclp} for several conclusions 
in the setting of compact Lie groups
concerning
summability and distribution of eigenvalues of operators on $L^{p}$-spaces once
we have information on their $r$-nuclearity. See also \cite{drtbvp:p1} for applications of the notion of nuclearity to boundary value problems. Kernel conditions on compact manifolds have been investigated in 
\cite{dr:suffkernel}, \cite{dr:sdk}.

\medskip

On the other hand, the variable exponent Lebesgue spaces are a generalisation of the classical Lebesgue spaces, replacing the constant exponent $p$ by a variable exponent function $p(x)$. Variable exponent Lebesgue spaces were introduced by Orlicz \cite{or:mod} in 1931 and some properties were further developed by Nakano in the 1950s \cite{nak:a1}, \cite{nak:a2} within the more general framework of modular spaces. Subsequently developments of modular spaces were carried out in the 1970s and 1980s by Hudzik, Musielak, Portnov \cite{hu:so1}, \cite{hu:so2}, \cite{hu:so3}, \cite{mus:or}, \cite{por:or}.

A more specific study of variable Lebesgue spaces only appears in 1961 with the work of Tsenov \cite{tse:ve} who independently discovered those spaces and later in the works of Sharapudinov \cite{sha:ve1},  \cite{sha:ve2},  \cite{sha:ve3},  \cite{sha:ve4} and Zhikov \cite{zh:pc}, \cite{zh:pl}, \cite{zh:os}.  Further, the development of the analysis of many problems on those spaces has  
  been of great interest in the last decades as has been exhibited in the recent book
  \cite{di:book}, \cite{lpv:cuf}, \cite{lpv:cufrw} and the literature therein. 

We now briefly recall the definition of variable exponent Lebesgue spaces and we refer the reader to \cite{di:book} and \cite{lpv:cuf} for the basic properties of such spaces. Let $(\Omega,\mathcal{M},\mu)$ be a $\sigma$-finite, complete measure space. We
define $\mathcal{P}(\Omega,\mu)$ to be the set of all $\mu$-measurable functions $p: \Omega\rightarrow [1,\infty].$ The functions in
$\mathcal{P}(\Omega,\mu)$ are called variable exponents on $\Omega$. We define 
$$
p^+=p_{\Omega}^+:={\rm ess}\sup_{x\in \Omega}p(x), \quad p^-=p_{\Omega}^-:={\rm ess}\inf_{x\in\Omega}p(x).
$$
If $p^+<\infty$, then $p$ is called a {\em bounded variable exponent}.
 If $f:\Omega\rightarrow\er$ is a measurable function we define the {\em modular} associated with $p=p(\cdot)$ by 
\[\rho_{p(\cdot)}(f):=\int\limits_{\Omega}|f(x)|^{p(x)}d\mu(x),
\]
and
\[\|f\|_{\vel(\mu)}:=\inf\{\lambda >0:\rho_{p(\cdot)}(f/\lambda)\leq 1\}.\]
The resulting spaces $L^{p(\cdot)}(\mu)$ of measurable functions such that $\|f\|_{\vel(\mu)}<\infty$ are Banach spaces and enjoy many properties similar to the classical Lebesgue $L^p$ spaces. 
For example, we will often make use of the following modification of H\"older inequality which becomes affected by factor 2: let
$p,q,s\in \mathcal{P}(\Omega,\mu)$ be such that
$$
\frac{1}{s(x)}=\frac{1}{p(x)}+\frac{1}{q(x)}
$$
holds for $\mu$-almost every $x\in\Omega$. Then we have
\begin{equation}\label{EQ:Holder}
\|fg\|_{L^{s(\cdot)}(\mu)}\leq  2\, \|f\|_{L^{p(\cdot)}(\mu)} \|g\|_{L^{q(\cdot)}(\mu)}.
\end{equation}
We refer to \cite[Lemma 3.2.20]{di:book} for a more detailed statement.

At the same time, there are some exceptions and differences to the classical theory, for instance the Young inequality fails in the variable exponent case, a fact proved in 1991 by Kov{\'a}{\v c}ik and R{\'a}kosn{\'i}k (cf. \cite{kr:che}) and essentially due to the loss of boundedness of translation operators on  $L^{p(\cdot)}$ spaces (see also \cite{lpv:cuf}, Theorem 5.19). If the variable exponent $p(\cdot)$ is bounded the space $L^{p(\cdot)}(\mu)$ is separable and if we denote by $p'(\cdot)$ the variable exponent defined pointwise by
\[
\frac{1}{p(x)}+\frac{1}{p'(x)}=1,
\]
then $(L^{p(\cdot)}(\mu))'=L^{p'(\cdot)}(\mu)$, where the identity refers to the associate space and not necessarily to the isometric dual space. Moreover, if $1<p^-\leq p^+<\infty$ the space $L^{p(\cdot)}(\mu)$ is reflexive.  For the study of the approximation property we will restrict to consider bounded variable exponents due to the density of the simple functions in $L^{p(\cdot)}$ in that case.

In this work we are going to establish the bounded approximation property for variable exponent Lebesgue spaces, study the concept of nuclearity
 on such spaces and apply it to trace formulae such as the 
Grothendieck-Lidskii formula and the analysis of pseudo-differential operators on the torus. 
 
\section{Bounded approximation property for variable exponent Lebesgue spaces}

In  this section we will prove that the  variable exponent spaces $L^{p(\cdot)}(\mu)$ satisfy the bounded approximation property.
\medskip

In the rest of this section we will assume that our measure space $(\Omega,\mathcal{M},\mu)$ is $\sigma$-finite and complete. 
We will also assume that the exponent $p(\cdot)$ is bounded since only in such case the simple functions are dense in $L^{p(\cdot)}(\mu)$ (cf. \cite{di:book}, Corollary 3.4.10). 

We shall now formulate some preparatory lemmata useful for the proof of the bounded  approximation property. Let $I$ be a countable set of indices endowed with the counting measure $\nu$. For $p\in\mathcal{P}(I,\nu)$, we will denote by $\ell^{p(\cdot)}(I)$ or simply by 
$\ell^{p(\cdot)}$ the  corresponding variable exponent Lebesgue space whose norm is given by
\[
\|h\|_{\ell^{p(\cdot)}}=\inf\{\lambda >0:\sum\limits_{k\in I}\left|\frac{h_k}{\lambda}\right|^{p_k}\leq 1\}.
\]
Given a Banach space $\bba$ and $u\in\bba, z\in\bba '$, we will also denote by $\langle u,z\rangle _{\bba, \bba'}$, or simply 
by $\langle u,z\rangle $, the valuation $z(u)$. 
\begin{lem}\label{simp1} Let $\bba$ be a Banach space and $q\in\mathcal{P}(I,\nu)$. Let $(u_i)_{i\in I}, (v_i)_{i\in I}$ be sequences in $\bba ', \bba$ respectively such that
\[\|\langle x,u_i\rangle\|_{\ell^{q(\cdot)}}, \|\langle v_i,z\rangle\|_{\ell^{q'(\cdot)}}\leq 1,\,\,{\mbox { for }} \|x\|_{\bba}, \|z\|_{\bba '}\leq 1.\]
Then the operator $T=\sum\limits_{i\in I}u_i\otimes v_i$ from $\bba$ into $\bba$ is well defined, bounded and satisfies
$\|T\|_{\mathcal{L}(\bba)}\leq 2$.
\end{lem}
\begin{proof} Let $N\subset I$ be a finite subset of $I$. Let us write $T_N:=\sum\limits_{i\in N}u_i\otimes v_i$.
 It is clear that $T_N$ is well defined. Moreover $T_N$ is a bounded finite rank operator. Now, since 
$T_Nx=\sum\limits_{i\in N}\langle x,u_i\rangle v_i$, we observe that for $x\in\bba, z\in\bba '$ such that 
$\|x\|_{\bba}, \|z\|_{\bba '}\leq 1$, applying the H\"older inequality \eqref{EQ:Holder} for variable exponent spaces
we obtain  
$$
|\langle T_Nx,z\rangle|\leq \sum\limits_{i\in N}|\langle x,u_i\rangle||\langle v_i,z\rangle|
 \leq 2\|\langle x,u_i\rangle\|_{\ell^{q(\cdot)}}\|\langle v_i,z\rangle\|_{\ell^{q'(\cdot)}}
\leq  2.
$$
 Therefore $T=\lim\limits_NT_N$ exists in $\mathcal{L}(\bba)$ and $\|T\|_{\mathcal{L}(\bba)}\leq 2$.
\end{proof}

\begin{lem}\label{simp2} Let $\bba_1, \bba_2$ be Banach spaces and $(L_i)_i$ a net contained in $\mathcal{L}(\bba_1,\bba_2)$ such that
 for every $x\in\bba_1$, $\lim\limits_{i}L_ix=Lx$ for some $Lx\in\bba_2$. Then
$L\in\mathcal{L}(\bba_1,\bba_2)$, $\|L_i\|_{\mathcal{L}(\bba_1,\bba_2)}\leq M$ for some $M>0$ and $L_i$ converge to $L$ in the topology of
 uniform convergence on compact sets.
\end{lem}
\begin{proof} The fact that $\|L_i\|_{\mathcal{L}(\bba_1,\bba_2)}\leq M$ and $L\in\mathcal{L}(\bba_1,\bba_2)$ follows from the 
uniform boundedness principle. For the rest, let $K\subset\bba_1$ be compact, $\epsilon >0$ and $M\geq 1$ such that $\|L_i\|_{\mathcal{L}(\bba_1,\bba_2)}\leq M$. Let $\{x_1,\dots,x_n\}\subset K$ be such that $K\subset \bigcup\limits_{j=1}^nB(x_j,\frac{\epsilon}{3M})$. If $i$ is large enough
we have $\|Lx_j-L_ix_j\|_{\bba_2}<\frac{\epsilon}{3M}$ for all $1\leq j\leq n$. Let $x\in K$ and we pick $j_0$ such that 
$\|x-x_{j_0}\|_{\bba_1}<\frac{\epsilon}{3M}$. Then
\[\|Lx-L_ix\|_{\bba_2}\leq\|Lx-Lx_{j_0}\|_{\bba_2}+\|Lx_{j_0}-L_ix_{j_0}\|_{\bba_2}+\|L_ix_{j_0}-L_ix\|_{\bba_2}<\epsilon .\]
Therefore $L_i$ converge to $L$ uniformly on compact sets.
\end{proof}

As a consequence we obtain:
\begin{cor}\label{simp3} Let $\bba$ be a Banach space. If there is a net $(L_i)_i$ contained in $\mathcal{F}(\bba)$ such that
 $\sup\limits_{i}\|L_i\|\leq M<\infty$ and $\lim\limits_{i}L_ix=x$ for every $x\in\bba$, then
$\bba$ has the bounded approximation property with constant $M$.
\end{cor}

We can now prove the main result of this section:

\begin{thm} \label{app12} Let $p\in\mathcal{P}(\Omega,\mu)$ be a bounded variable exponent. Then, the variable exponent Lebesgue space $L^{p(\cdot)}(\mu)$ has the bounded approximation property.
\end{thm}

Since the H\"older inequality \eqref{EQ:Holder} in variable Lebesgue spaces $L^{p(\cdot)}$ holds with constant $2$, we obtain the bounded approximation property in this setting, rather than the metric approximation property valid for the usual $L^p$-spaces. 


However, the bounded approximation property implies the metric approximation property if the space is reflexive. In our case, this happens if $\Omega\subset\Rn$ is open and $1<p^{-},  p^{+}<\infty$ (cf. \cite{di:book}), in which case $L^{p(\cdot)}(\Omega)$ has the metric approximation property. This gives a small change to Theorem 3.1 in \cite{DRW-J-spectr-theory}, with the rest of \cite{DRW-J-spectr-theory} unchanged.

\begin{proof}[Proof of Theorem \ref{app12}]
We first consider the case when $p\in\mathcal{P}(\Omega,\mu)$ is a simple function and we write 
\[p(x)=\sum\limits_{j=1}^l p_j1_{\Omega_j}(x),\]
where $p_j>0$, the sets $\Omega_{j}$ are disjoint of finite measure, $1_{\Omega_{j}}$ denotes the characteristic function of the set $\Omega_{j}$.

Let $\mathfrak{P}=\{\Omega_1,\dots,\Omega_l\}$ be a finite family of disjoint measurable sets of finite positive measure. We denote
by \textbf{P} the collection of such families.  To a $\mathfrak{P}\in\textbf{P}$ we associate a  finite rank  operator $L_{\mathfrak{P}}$ from $L^{p(\cdot)}$ into $L^{p(\cdot)}$ defined by
\beq L_{\mathfrak{P}}f:=\sum\limits_{k=1}^l\mu(\Omega_{k})^{-1}\langle f, 1_{\Omega_{k}} \rangle _{\vel, \velp}1_{\Omega_{k}},\label{op1a}\eq 
where $1_{\Omega_{k}}$ denotes the characteristic function of the set $\Omega_{k}$.

We observe that $L_{\mathfrak{P}}f$ is well defined since $0<\mu(\Omega_{k})<\infty$ for $1\leq k\leq l$ and the duality $\langle\cdot ,\cdot \rangle _{\vel, \velp}$ is well defined by using the H\"older inequality for variable exponent spaces,
see \eqref{EQ:Holder}.

In the collection $\bf{P}$ we define the partial order $\mathfrak{P}_1\leq \mathfrak{P}_2$ if any set in $\mathfrak{P}_1$ is the union of sets in $\mathfrak{P}_2$. We also say that $\mathfrak{P}_2$ is {\em finer} than $\mathfrak{P}_1$ if $\mathfrak{P}_1\leq \mathfrak{P}_2$. 
This order begets a directed set. 

Let $\mathfrak{P}_1$ be the family of sets associated to the exponent $p(\cdot)$. By choosing a finer 
 $\mathfrak{P}$ we can rewrite the operator $L_{\mathfrak{P}}$ given by \eqref{op1a} in different ways which will be useful later on:  we define
\[
u_{k}:=\frac{1_{\Omega_{k}}}{\mu(\Omega_k)^{\frac{1}{p_k'}}}\, ,\,\,v_{k}:=\frac{1_{\Omega_{k}}}{\mu(\Omega_k)^{\frac{1}{p_k}}}\,\,, 
\]
so that we can write
\begin{align*}L_{\mathfrak{P}}=&\sum\limits_{k=1}^l\frac{1_{\Omega_{k}}}{\mu(\Omega_k)^{\frac{1}{p_k'}}}\otimes\frac{1_{\Omega_{k}}}{\mu(\Omega_k)^{\frac{1}{p_k}}}\\
=&\sum\limits_{k=1}^lu_{k}\otimes v_{k}\\
=&\sum\limits_{k=1}^l\frac{1}{\mu(\Omega_{k})}\left(1_{\Omega_{k}}\otimes 1_{\Omega_{k}}\right).
\end{align*}
We will prove that $\|L_{\mathfrak{P}}\|_{{\mathcal{L}(L^{p(\cdot)})}}\leq 2$ by applying   
Lemma \ref{simp1} in the case $\bba=\vel$, the finite families $u_k,v_k$ and $q=p(k)=p_k$. Let $f\in\vel$, $g\in\velp$ be such that $\|f\|_{\vel}, \|g\|_{\velp}\leq 1$. Then we have to show that
\[\|\langle f,u_{k}\rangle\|_{\ell^{p(\cdot)}}\leq 1
\; \textrm{ and } \; \|\langle v_{k},g\rangle\|_{\ell^{p'(\cdot)}}\leq 1 . \]
In order to prove the corresponding property for $f\in\vel$, it is enough to consider a   
simple function $f\in\vel$ such that $\|f\|_{\vel}\leq 1$. The general case follows then by a standard density argument.
 By redefining partitions, we can assume that $f$ can be written in the form
\[f(x)=\sum\limits_{k=1}^l\beta_{k}1_{\Omega_{k}}(x).\]
Now, for $\lambda>0$ we have
$$
\rho_{p(\cdot)}(f/\lambda)=\int\limits_{\Omega}\left|\frac{f(x)}{\lambda}\right|^{p(x)}dx
=\sum\limits_{k=1}^l\int\limits_{\Omega_k}\left|\frac{\beta_k}{\lambda}\right|^{p_k}dx
=\sum\limits_{k=1}^l \left|\frac{\beta_k}{\lambda}\right|^{p_k}\mu(\Omega_k).
$$
We also observe that 
\[\langle f,u_k\rangle =\beta_{k}\frac{\mu(\Omega_k)}{\mu(\Omega_k)^{\frac{1}{p_k'}}}=\beta_{k}\mu(\Omega_k)^{\frac{1}{p_k}}.\]
Hence
\begin{align*}\|\langle f,u_k\rangle\|_{\ell^{p(\cdot)}}=&\inf\{\lambda >0:\sum\limits_{k=1}^l
\left|\frac{\beta_k}{\lambda}\right|^{p_k}\mu(\Omega_k)\leq 1\}\\
=&\inf\{\lambda >0:\rho_{p(\cdot)}(f/\lambda)\leq 1\}\\
=&\|f\|_{\vel}\leq 1.
\end{align*}
We have shown that $\|\langle f,u_k\rangle\|_{\ell^{p(\cdot)}}\leq 1$, the proof of $\|\langle v_{k},g\rangle\|_{\ell^{p'(\cdot)}}\leq 1$ is similar and we omit it.  Hence $\|L_{\mathfrak{P}}\|_{{\mathcal{L}(\vel)}}\leq 2$. 
\medskip

We now consider the net of finite rank operators $(L_{\mathfrak{P}})_{{\mathfrak{P}}\geq \mathfrak{P}_1}$ and prove that 
\[\lim\limits_{\mathfrak{P}}L_{\mathfrak{P}}f=f\]
for every $f\in \vel$. It is enough to see this for $f$ simple by   the density of simple functions in $\vel$ (cf. \cite{di:book}, Corollary 3.4.10). Indeed, let us write $f$ in the form 
\[f(x):=\sum\limits_{m=1}^s\alpha_{m}1_{\widetilde{\Omega}_m}(x).\]
If we chose $\mathfrak{P}$ finer than $\mathcal{Q}_1=\{\widetilde{\Omega}_m:1\leq m\leq s\}$, then $L_{\mathfrak{P}}f=f$. Indeed, since the sets $\Omega_{k}$ are disjoint we have
\begin{align*}
L(1_{\Omega_j})=&\sum\limits_{k=1}^l\frac{1}{\mu(\Omega_{k})}\left(1_{\Omega_{k}}\otimes 1_{\Omega_{k}}\right)(1_{\Omega_j})\\
=&\sum\limits_{k=1}^l\frac{1}{\mu(\Omega_{k})}\left\langle 1_{\Omega_{j}} , 1_{\Omega_{k}}\right\rangle _{\velp, \vel}(1_{\Omega_k})\\
=&\frac{1}{\mu(\Omega_{j})}\left\langle 1_{\Omega_{j}} , 1_{\Omega_{j}}\right\rangle _{\velp, \vel}(1_{\Omega_j})\\
=&1_{\Omega_j}.
\end{align*}
Therefore, $L_{\mathfrak{P}}f=f$ for $\mathfrak{P}\geq\mathcal{Q}_1$ and thus $\lim\limits_{\mathfrak{P}}L_{\mathfrak{P}}f=f$ in $\vel$.
 We have actually proved  that $\vel$ satisfied the bounded approximation property by an application of Corollary \ref{simp3}.

By an additional argument we will obtain the desired property in the general case. We now consider a variable exponent $p(\cdot)$ such that $p^+<\infty$. Then, there exists an increasing sequence of simple functions $p^j(\cdot)$ such that $\lim\limits_{j}p^j(\cdot)=p(\cdot)$ a.e. For each $j$ we associate to a $\mathfrak{P}^j\in\textbf{P}$ an  
operator $L_{\mathfrak{P}^j}$ as in \eqref{op1a} which due to its form is also defined from $L^{p(\cdot)}(\mu)$ into $L^{p(\cdot)}(\mu)$. 
  The fact that $\lim\limits_{j}L_{\mathfrak{P}^j}f=f$ in $\vel$ follows as in the previous case. We claim that $\|L_{\mathfrak{P}^j}\|_{{\mathcal{L}(\vel)}}\leq 2$ which by an application of Corollary \ref{simp3} with $M=2$ will conclude the proof. Indeed, if $f$ is a simple function such that $\|f\|_{\vel}\leq 1$ and $\mathfrak{P}_1$ is its corresponding family in $\bf{P}$, we observe that by choosing $\mathfrak{P}^j\geq\mathfrak{P}_1$ we obtain $L_{\mathfrak{P}^j}f=f$.
\end{proof}

\section{Nuclearity on variable exponent Lebesgue spaces}

In this section we establish some basic properties for the kernels of nuclear operators on $\vel$ spaces. We also prove 
a characterisation of nuclear operators on $\vel$.

We start by proving a lemma giving  basic properties
 of a kernel corresponding to a nuclear operator on $\vel$ spaces when $\mu$ is a finite measure. In the rest of this section we shall consider two variable exponents $p(\cdot)\in\mathcal{P}(\Omega,\mu), q(\cdot)\in\mathcal{P}(\Xi,\nu) $ and the  
 variable exponent conjugate $p'(\cdot)$ of $p(\cdot)$ such that
$$\frac{1}{p(\cdot)}+\frac{1}{p'(\cdot)}=1.$$

\begin{lem}\label{l1}
Let $({\Omega},{\mathcal{M}},\mu)$ and
$({\Xi},{\mathcal{M}}',\nu)$ be two finite and complete measure spaces.
Let $f\in L^{p(\cdot)}(\mu)$, and $(g_n)_n,(h_n)_n$ be sequences in
$L^{q(\cdot)}(\nu)$
  and $L^{p'(\cdot)}(\mu)$, respectively, such that $\sum \limits_{n=1}^\infty \| g_n\|_{L^{q(\cdot)}} \|h_n\|_{L^{p'(\cdot)}}<\infty$. Then
\begin{itemize}
\item[(a)] The series $\sum\limits_{j=1}^{\infty} g_j(x)h_j(y)$
converges absolutely for a.e. $(x,y)$ and, consequently, 
${\displaystyle\lim\limits_n\sum\limits_{j=1}^n
        g_j(x)h_j(y)\,\mbox{ is finite for a.e.}\, (x,y)}.$

\item[(b)] For $k(x,y):={\displaystyle \sum\limits_{j=1}^{\infty} g_j(x)h_j(y)},$
we have $k\in L^1(\nu\otimes\mu)$.

\item[(c)] If $k_n(x,y)=\sum\limits_{j=1}^n
        g_j(x)h_j(y) $ then $\|k_n-k\|_{L^1(\nu\otimes\mu)}\rightarrow 0$.

\item[(d)] ${\displaystyle\lim\limits_n \int\limits_{\Omega}\left(\sum\limits_{j=1}^n  g_j(x)h_j(y)\right)f(y)d\mu(y)= \int\limits_{\Omega} \left(\sum\limits_{j=1}^\infty  g_j(x)h_j(y)\right)f(y)d\mu(y),}$
\end{itemize}
\noindent \textit{ for  a.e} $x$.
\end{lem}
\begin{proof} We first write $\tilde k_n(x,y):=\sum\limits_{j=1}^n g_j(x)h_j(y)f(y)$ and note that since $\nu$ is finite then $\|1\|_{L^{q'(\cdot)}(\nu)}<\infty$. Indeed, $\int\limits_{\Xi}|1/\lambda|^{q'(x)}d\nu(x)=\nu(\Xi)<\infty ,$ for $\lambda=1$. Now, by applying the H{\"o}lder inequality \eqref{EQ:Holder} which is affected by the factor 2 in the setting of variable exponents we obtain 
\begin{align*}
\int\limits_{\Omega}\!\int\limits_{\Xi}|\tilde k_n(x,y)|d\nu(x)d\mu(y)\leq&\int\limits_{\Omega}\!\int\limits_{\Xi}\sum\limits_{j=1}^n
|g_j(x)
  h_j(y)f(y)|d\nu(x)d\mu(y)\\
\leq &\sum\limits_{j=1}^n \int\limits_{\Xi}
|g_j(x)|d\nu(x)\int\limits_{\Omega}|h_j(y)||f(y)|d\mu(y)\\
\leq& 2\|1\|_{L^{q'(\cdot)}(\nu)}\|f\|_{L^{p(\cdot)}(\mu)}\sum \limits_{j=1}^n \| g_j\|_{L^{q(\cdot)}(\nu)}
\|h_j\|_{L^{p'(\cdot)}(\mu)}\\
\leq & M<\infty \; \mbox{for all}\; n.
\end{align*}
Hence $\|\tilde k_n\|_{{L}^1(\nu\otimes\mu)}\leq M$ for all $n$. We now consider a sequence $(s_n)$ defined by
 $$s_n(x,y):=\sum\limits_{j=1}^n |g_j(x)h_j(y)f(y)|,$$ 
 which is increasing in   $L^1(\nu\otimes\mu)$ and satisfies 
 \[
 \sup\limits_n \int\int |s_n(x,y)|d\mu(x)d\mu(y) \leq M<\infty.
 \] By an application of the monotone convergence theorem, the limit 
 $$
 s(x,y)=\lim\limits_n s_n(x,y)
 $$
 exists for a.e. $(x,y)$ and  $s\in L^1(\nu\otimes\mu)$. Moreover, since $f=1\in\vel(\mu)$ and the fact that 
$|k(x,y)|\leq s(x,y)$ we deduce (a) and (b). 

The part (c) can be deduced
by using the Lebesgue dominated convergence theorem applied to the sequence
$(\tilde k_n)$ dominated by $s(x,y)$, and setting $f\equiv 1$, in which case $\tilde k_n=k_n$.

For the part (d) we observe that
with $\tilde k_n(x,y)=\sum\limits_{j=1}^n g_j(x)h_j(y)f(y)$, we have
$|\tilde k_n(x,y)|\leq s(x,y)$ for all $n$ and every $(x,y)$. From the
fact that $s\in L^1(\nu\otimes\mu)$ we obtain that
$s(x,\cdot)\in L^1(\mu)$ for a.e $x$. Then (d) is obtained from Lebesgue dominated
convergence theorem.
\end{proof}

\begin{rem} We observe that the condition of finiteness of the measures in the lemma above is crucial to obtain 
$k=k(x,y)\in
L^1(\nu\otimes\mu)$. For instance, let $\Omega=\Xi=\ern$, $\mu=\nu$ be the Lebesgue measure and 
$p=p(\cdot), q=q(\cdot)$ constant exponents such that $1\leq p,q<\infty$. 
Then by using the fact that $p'>1$, we define $k(x,y):=g(x)h(y)$, with $g\in L^{q}(\mu)\setminus \{0\},h\in L^{p'}(\mu)\setminus L^{1}(\mu)$. Then  
$$
\int\limits_{\ern}\int\limits_{\ern}|k(x,y)|d\mu(x)d\mu(y)=\int\limits_{\ern}|g(x)|d\mu(x)\int\limits_{\ern}|h(y)|d\mu(y)=\infty.
$$
\end{rem}

We can now formulate a characterisation of $r$-nuclear operators on variable exponent Lebesgue spaces for finite measure spaces. 

\begin{thm}\label{ch1} 
Let $({\Omega},{\mathcal{M}},\mu)$ and
$({\Xi},{\mathcal{M}}',\nu)$ be two complete and finite measure spaces. Let $0<r\leq 1$. Then $T$ is
  $r$-nuclear operator from $\vel(\mu)$ into $\velq(\nu)$ if and only if there exist a sequence
 $(g_n)$ in $\velq(\nu)$, and a sequence $(h_n)$ in $\velp(\mu)$ such that $\sum \limits_{n=1}^\infty \|
 g_n\|_{\velq(\nu)}^r\|h_n\|_{\velp(\mu)}^r<\infty$, and such that for all $f\in\vel(\mu)$, we have
\[Tf(x)=\int\limits_{\Omega}\left(\sum\limits_{n=1}^{\infty}
  g_n(x)h_n(y)\right)f(y)d\mu(y), \,\,\mbox{for a.e } x.\]
\end{thm}
\begin{proof} We will assume that $r=1$. The case $0<r<1$ follows by inclusion. 
Let $T$ be a nuclear operator from $\vel(\mu)$ into $\velq(\nu)$. Then
there exist sequences $(g_n)$ in $\velq(\nu)$, $(h_n)$ in
$\velp(\mu)$ such that
 $\sum \limits_{n=1}^\infty \| g_n\|_{\velq(\nu)}
 \|h_n\|_{\velp(\mu)}<\infty$ and
\[Tf= \sum\limits_n \left <f,h_{n}\right>g_n .\]
\noindent Now
 \[Tf=\sum\limits_n \left <f,h_n\right>g_n =\sum\limits_n \left(\int\limits_{\Omega}
   h_n(y)f(y) d\mu(y)\right)g_n  \, ,  \]
where the sums converge with respect to the $\velq(\nu)$-norm. There exist
(cf. \cite{di:book}, Lemma 3.2.10)
two sub-sequences $(\widetilde{g}_n)$ and $(\widetilde{h}_n)$ of $(g_n)$ and
$(h_n)$ respectively such that

\[(Tf)(x)=\sum\limits_n \left <f,\widetilde{h}_n\right>\widetilde{g}_n (x)=\sum\limits_n \left(\int\limits_{\Omega}\widetilde{h}_n(y)f(y) d\mu(y)\right)\widetilde{g}_n (x),  \,\,\mbox{ for a.e } x. \]
\noindent Now taking into account that the pair $\left((\widetilde{g}_n),(\widetilde{h}_n) \right)$
satisfies
\[\sum \limits_{n=1}^\infty \|\widetilde{ g}_n\|_{\velq(\nu)}
 \|\widetilde{h}_n\|_{\velp(\mu)}<\infty \, ,\]
  and by applying Lemma \ref{l1} (d), it follows that
\begin{align*}
 \sum\limits_{n=1}^{\infty} \left(\int\limits_{\Omega}\widetilde{h}_n(y)f(y)
d\mu(y)\right)\widetilde{g}_n(x)=&\lim\limits_n  \sum\limits_{j=1}^n
\left(\int\limits_{\Omega}\widetilde{h}_j(y)f(y) d\mu(y)\right)\widetilde{g}_j(x)\\
=&\lim\limits_n \int\limits_{\Omega} \left(\sum\limits_{j=1}^n
\widetilde{g}_j(x)\widetilde{h}_j(y)f(y)\right)d\mu(y)\\
=&\int\limits_{\Omega}\left(\sum\limits_{n=1}^{\infty}
  \widetilde{g}_n(x)\widetilde{h}_n(y)\right)f(y)d\mu(y)\, , \,\mbox{ for a.e } x.  \end{align*}
Conversely, let us assume that there exist sequences
$(g_n)_n$ in $\velq(\nu)$, and $(h_n)_n$ in $\velp(\mu)$ such 
that $\sum \limits_{n=1}^\infty \| g_n\|_{\velq(\nu)}
 \|h_n\|_{\velp(\mu)}<\infty$, and for all $f\in\vel(\mu)$
\[Tf(x)=\int\limits_{\Omega}\left(\sum\limits_{n=1}^{\infty}
  g_n(x)h_n(y)\right)f(y)d\mu(y)\, , \,\,\mbox{ for a.e } x.\]
The  Lemma \ref{l1} (d) gives
\begin{align*}
 \int\limits_{\Omega}\left(\sum\limits_{n=1}^{\infty}
   g_n(x)h_n(y)\right)f(y)d\mu(y)&=&\lim\limits_n
 \int\limits_{\Omega}\left(\sum\limits_{j=1}^n  g_j(x)h_j(y)f(y)\right)d\mu(y)\\
  &=&\lim\limits_n  \sum\limits_{j=1}^n \left(\int\limits_{\Omega} h_j(y)f(y)
   d\mu(y)\right)g_j(x)\\
   &=& \sum\limits_n \left(\int\limits_{\Omega} h_n(y)f(y)
   d\mu(y)\right)g_n(x)\\
&=&\sum\limits_n \left <f,h_n\right>g_n
 (x)=(Tf)(x) \, ,\,\,a.e.\, x.
\end{align*}
To prove that $Tf=\sum\limits_n \left
<f,h_n\right>g_n $ in $\velq(\nu)$ we let $s_n:=
\sum\limits_{j=1}^n\left<f,h_j\right>g_j$, then $(s_n)_n$ is a sequence in
$\velq(\nu)$ and
$$|s_n(x)|\leq \|f\|_{\vel(\mu)}\sum\limits_{j=1}^n
\|h_j\|_{\velp(\mu)}|g_j(x)|$$
\[\leq\|f\|_{\vel(\mu)}\sum\limits_{j=1}^\infty\|h_j\|_{\velp(\mu)}|g_j(x)|=:\gamma(x),\, \mbox{ for all }n.\]
Moreover, $\gamma$ is well defined and $\gamma\in \velq(\nu)$
since it is the increasing limit of the sequence
$(\gamma_n)_n=(\|f\|_{\vel(\mu)}\sum\limits_{j=1}^n
\|h_j\|_{\velp(\mu)}|g_j(x)|)_n$ of $\velq(\nu)$ functions and
$$
\|\gamma_n\|_{\velq(\nu)}\leq\|f\|_{\vel}\sum\limits_{j=1}^\infty\|h_j\|_{\velp(\mu)}\|g_j\|_{\velq(\nu)}\leq M<\infty.
$$
By  the monotone convergence theorem we see that
$\gamma\in \velq(\nu)$. Finally, applying the Lebesgue dominated
convergence theorem we deduce that $s_n\rightarrow Tf$ in $\velq(\nu)$.
\end{proof}

In the sequel we also establish a characterisation of $r$-nuclear operators for $\sigma$-finite measures. 
In order to get an analogue of the finite measures setting we first generalise Lemma \ref{l1}. 

\begin{lem}\label{l1ab} Let $({\Omega},{\mathcal{M}},\mu)$ and
$({\Xi},{\mathcal{M}}',\nu)$ be two $\sigma$-finite, complete measure spaces.
Let $f\in L^{p(\cdot)}(\mu)$, and $(g_n)_n,(h_n)_n$ be sequences in
$L^{q(\cdot)}(\nu)$
  and $L^{p'(\cdot)}(\mu)$, respectively, such that $\sum \limits_{n=1}^\infty \| g_n\|_{L^{q(\cdot)}} \|h_n\|_{L^{p'(\cdot)}}<\infty$. Then
 the parts (a) and (d) of Lemma \ref{l1} hold.
\end{lem}

\begin{proof} (a) There exist two sequences
$(\Omega_k)_k$ and $(\Xi_j)_j$ of disjoint measurable subsets of $\Omega$ and
$\Xi$ respectively such that $\bigcup_k\Omega_k=\Omega,$ $\bigcup_j\Xi_j=\Xi$ and for all $j,k$
\[\mu(\Omega_k), \nu(\Xi_j)<\infty .\]

We now consider the  respective restricted measure spaces $(\Omega_k,{\mathcal
   M}_k,\mu_k)$ and also  $(\Xi_j,{\mathcal
   M '}_j,\nu_j)$ that we obtain by restricting $\Omega$ to
 $\Omega_k$, and $\Xi$ to
 $\Xi_j$ for every $k,j$, and restricting the functions $g_n$
 to $\Xi_j$, and  $h_n$ to $\Omega_k$ for each $n$. Then, for all $k,j$
\[\sum \limits_{n=1}^\infty \| g_n\|_{\velq(\nu_j)} \|h_n\|_{\velp(\mu_k)}<\infty. \]
\noindent By Lemma \ref{l1} (a) it follows that 
 $\sum\limits_{j=1}^{\infty} g_j(x)h_j(y)$ converges absolutely for  a.e $(x,y)\in
\Xi^j\times \Omega^k $. Hence
 $\sum\limits_{j=1}^{\infty}
g_j(x)h_j(y)$ converges absolutely for almost every  $(x,y)\in \Xi\times \Omega $. This proves part (a).\\
\noindent From the part (a) the series
$\sum\limits_{j=1}^{\infty} g_j(x)h_j(y)f(y)$ converges absolutely
for a.e. $(x,y)\in \Xi\times\Omega$, the part (d) follows from the Lebesgue dominated
convergence theorem applied as in the ``only if" part of the proof
of Theorem \ref{ch1}.
\end{proof}

We are now ready to give the main result of this section, the extension of Theorem \ref{ch1} 
to the setting of $\sigma$-finite measures.

\begin{thm}\label{ch2}  
Let $({\Omega},{\mathcal{M}},\mu)$ and
$({\Xi},{\mathcal{M}}',\nu)$ be $\sigma$-finite complete measure spaces. Let $0<r\leq 1$. Then $T$ is
  $r$-nuclear operator from $\vel(\mu)$ into $\velq(\nu)$ if and only if there exist a sequence
 $(g_n)$ in $\velq(\nu)$, and a sequence $(h_n)$ in $\velp(\mu)$ such that $\sum \limits_{n=1}^\infty \|
 g_n\|_{\velq(\nu)}^r\|h_n\|_{\velp(\mu)}^r<\infty$, and such that for all $f\in\vel(\mu)$ we have
\[
Tf(x)=\int\limits_{\Omega}\left(\sum\limits_{n=1}^{\infty}
  g_n(x)h_n(y)\right)f(y)d\mu(y), \,\,\mbox{for a.e } x.
\]

Moreover, if $\Omega=\Xi$, $\mu=\nu$, $p(\cdot)=q(\cdot)$, $p^+<\infty$, and $T$ is $r$-nuclear in 
$\mathcal{L}(\vel(\mu))$, then  
\[\Tr(T)=\sum\limits_{n=1}^{\infty} \left< g_n,h_n \right>=\int\limits_{\Omega}\sum\limits_{n=1}^{\infty}g_n(x)h_n(x)d\mu.\]
\end{thm}

\begin{proof} 
For the characterisation it is enough to consider the case $r=1$. But that characterisation
 now follows from the same lines of the proof of Theorem \ref{ch1} by replacing the references to part (d) of Lemma \ref{l1} by part (d) of Lemma \ref{l1ab}.  On the other hand, since $p^+<\infty$ the bounded approximation property holds and the trace is well defined.  We observe that by definition \eqref{EQ:Trace} we can use the
	sequences $g_n, h_n$ to calculate the trace, which gives $\Tr(T)=\sum\limits_{n=1}^{\infty} \left< g_n,h_n \right>$. Moreover, 
	the kernel $k(x,y)=\sum\limits_{n=1}^{\infty} g_n(x)h_n(y)$ is well defined on the diagonal since
	for $p(\cdot)=q(\cdot)$, we have
\begin{align*}|k(x,x)|\leq & \int\limits_{\Omega}\sum\limits_{n=1}^{\infty} |g_n(x)h_n(x)|d\mu\\
 =&\sum\limits_{n=1}^{\infty}\int\limits_{\Omega} |g_n(x)h_n(x)|d\mu\\
\leq &\sum\limits_{n=1}^{\infty}\|
 g_n\|_{\velq(\nu)}\|h_n\|_{\velp(\mu)}<\infty.
\end{align*}
Therefore, $k(x,x)\in L^1(\mu)$, $k(x,x)$ is finite for a.e. $x$ and 
$$\sum\limits_{n=1}^{\infty} \left< g_n,h_n \right>=\int\limits_{\Omega}\sum\limits_{n=1}^{\infty}g_n(x)h_n(x)d\mu,$$
completing the proof.
\end{proof}

\section{Nuclearity on $\vel$ of operators on the torus }
\label{SEC:r-nuc-plapl}

In practice the application of the concept of nuclearity requires an underlying discrete analysis. A source of problems where this situation arises  in a natural way is the analysis of operators on compact Lie groups due to the discreteness of the unitary dual. More generally, a discrete Fourier analysis can be associated to a compact manifold as well as a notion of global symbol as developed in \cite{dr14a:fsymbsch}, \cite{dr:suffkernel}. 

In this section we apply the concept of $r$-nuclearity on variable exponent Lebesgue spaces to the study of periodic operators
on $\Rn$ which we can realise as
operators on the torus $\mathbb{T}^n$, and we point out that all the results in this section have suitable extensions to the setting of compact Lie groups. Recent results on the nuclearity
 on Lebesgue spaces on compact Lie groups and Grothendieck-Lidskii formulae have been obtained in \cite{dr13a:nuclp}. The trace formulas that we establish here are expressed in terms of global toroidal symbols.  
 We first recall some notations and definitions for the Fourier analysis on the torus and the toroidal quantization. The toroidal quantization has been analysed extensively in \cite{Ruzhansky-Turunen-JFAA-torus} 
and \cite{RT:nfao, rt:book}, following the initial analysis in \cite{Ruzhansky-Turunen:torus-OT-2007}. 
\medskip

We denote the $n$-dimensional torus by $\Tn=\Rn/ \Zn$. Its unitary dual can be 
described as $\widehat{\Tn}\simeq \Zn$, and the collection $\{ \xi_k(x)=e^{2\pi i x\cdot k}\}_{k\in\Zn}$ 
is an orthonormal basis of $L^2(\Tn)$. We will use the notation $\langle \xi\rangle :=1+|\xi|$, where $|\cdot|$ denotes the euclidean norm.


\begin{defn} 
Let us denote by $\mathcal{S}(\zn)$ the space of {\em rapidly decaying} functions $\phi:\zn\rightarrow\mathbb{C}$. That is,
$\varphi\in \mathcal{S}(\zn) $ if for any $M>0$ there exists a constant $C_{\varphi,M}$ such that
$$|\varphi (\xi)|\leq C_{\varphi,M}\langle \xi\rangle^{-M}$$
 holds for all $\xi\in\zn$. The topology on $\mathcal{S}(\zn)$ is defined by the seminorms $p_k$, where $k\in\mathbb{N}_0$ and $p_k(\varphi)=\sup\limits_{\xi\in\zn}\langle \xi\rangle^{k}|\varphi(\xi)|. $
\end{defn}

In order to define the class of symbols that we will use, let us recall the definition of the Fourier transform on the torus for a function $f$ in $C^{\infty}(\tn)$ given by
\[(\efet f)(\xi)=\widehat{f}(\xi)=\int_{\tn}e^{-2\pi ix\cdot \xi}f(x)dx.\]
One can prove that 
\[\efet: C^{\infty}(\tn)\rightarrow \mathcal{S}(\zn)\]
is a continuous bijection. The reconstruction formula of $f$ in the form of a discrete integral or sum over the dual group $\zn$
is the Fourier series
\[f(x)=\sum_{\xi\in\zn }e^{2\pi ix\cdot \xi}(\efet f)(\xi).\]

A corresponding operator is associated to a symbol $\sigma(x,\xi)$ which will be called a periodic pseudo-differential operator or the operator given by the toroidal quantization: 
\beq  T_{\sigma}f(x)=\sum\limits_{\xi\in\zn}
  e^{2\pi i x\cdot\xi}\sigma(x,\xi)(\efet f)(\xi) ,\eq
 which can also be written as 
\beq  \label{EQ:Tq}
T_{\sigma}f(x)=\sum\limits_{\xi\in\zn}\int\limits_{\tn}
  e^{2\pi i(x-y)\cdot\xi}\sigma(x,\xi)f(y)dy.
 \eq
  
We refer to \cite{Ruzhansky-Turunen-JFAA-torus} for an extensive analysis of such toroidal quantization.


In the rest of this section we will consider $\tn$ endowed with the Borel $\sigma$-algebra and the Lebesgue measure so that we will just write $\mathcal{P}(\tn)$ 
 to denote the corresponding class of variable exponents. 

\begin{thm} 
Let $p(\cdot)\in\mathcal{P}(\tn)$ and $0<r\leq 1$. Let $\sigma(x,\xi)$ be a symbol such that
\[
\sum\limits_{\xi\in\zn}\|\sigma(\cdot,\xi)\|_{\velp}^r<\infty.
\]
Then $T_{\sigma}$ is $r$-nuclear from $\vel$ to $\velq$ for all $q(\cdot)\in\mathcal{P}(\tn)$. If $p^+<\infty$ and $q(\cdot)=p(\cdot)$, then $T_{\sigma}$ is $r$-nuclear on $\vel(\tn)$ and   
\beq \label{fftra}
\Tr(T_{\sigma})=\int_{\tn}\sum\limits_{\xi\in\zn} \sigma(x,\xi)dx.
\eq
In particular, if additionally $r\leq \frac 23$, then    
\beq \label{fftra2}
\Tr(T_{\sigma})=\int_{\tn}\sum\limits_{\xi\in\zn} \sigma(x,\xi)dx=\sum\limits_{j=1}^{\infty}\lambda_j,
\eq
where $\lambda_j\,\, (j=1,2,\dots)$ are the eigenvalues of $T_{\sigma}$ on $\vel(\tn)$ with multiplicities taken into account.
\end{thm}

\begin{proof} We observe that for a pseudo-differential operator $T_{\sigma}$ of the form \eqref{EQ:Tq} 
its kernel can be formally written in the form
\[k(x,y)=\sum\limits_{\xi\in\Zn} e^{2\pi i(x-y)\cdot \xi} \sigma(x,\xi).\]
We write $g_{\xi}(x)=e^{2\pi ix\cdot \xi} \sigma(x,\xi),\,\, h_{\xi}(y)=e^{-2\pi iy\cdot \xi}$. 
Now, $\|h_{\xi}(\cdot)\|_{\velq}=\|1\|_{\velq}<\infty$ since the measure is finite. Hence
\[\sum\limits_{\xi\in\zn}\|\sigma(\cdot,\xi)\|_{\velp}^r\|h_{\xi}(\cdot)\|_{\velq}^r=\|h_{\xi}(\cdot)\|_{\velq}^r\sum\limits_{\xi\in\zn}\|\sigma(\cdot,\xi)\|_{\velp}^r<\infty,\]
for all $q(\cdot)\in\mathcal{P}(\tn)$. An application of Theorem \ref{ch2} yields the $r$-nuclearity of $T_{\sigma}$ from $\vel$ to $\velq$ for all $q(\cdot)\in\mathcal{P}(\tn)$.  The formula \eqref{fftra} for the trace also follows from Theorem \ref{ch2} since $g_{\xi}(x)h_{\xi}(x)=\sigma(x,\xi).$
 The formula \eqref{fftra2} follows from  \eqref{fftra} and Grothendieck's Theorem.
\end{proof}
As an application, we will consider the composition of a multiplication operator with a multiplier
(an operator with symbol depending only on $\xi$) on the torus $\tn$.  Given a measurable function $\alpha$ on $\tn$,
we take the symbols $\alpha(x)$ and  $\sigma(\xi)$, the corresponding multiplication is the operator denoted by 
$\alpha T_{\sigma}$ given by $\alpha T_{\sigma}f=\alpha\sigma(D)f$ on $\tn$. 

\begin{cor} 
Let $p(\cdot)\in\mathcal{P}(\tn)$. Let $0<r\leq 1$, $\alpha\in\velp$, and let $\sigma(\xi)$ be a symbol such that
\[\sum\limits_{\xi\in\zn}|\sigma(\xi)|^r<\infty.\]
Then $\alpha T_{\sigma}$ is $r$-nuclear from $\vel$ to $\velq$ for all $q(\cdot)\in\mathcal{P}(\tn)$. 
If $p^+<\infty$ and $q(\cdot)=p(\cdot)$, then 
$\alpha T_{\sigma}$ is $r$-nuclear on $\vel(\tn)$ and   
\[
\Tr(\alpha T_{\sigma})=\int_{\tn}\alpha(x)dx\cdot\sum\limits_{\xi\in\zn} \sigma(\xi).
\]
If additionally  $r\leq \frac 23$, then 
\[
\Tr(\alpha T_{\sigma})=\int_{\tn}\alpha(x)dx\cdot\sum\limits_{\xi\in\zn} \sigma(\xi)=\sum\limits_{j=1}^{\infty}\lambda_j,
\]
where $\lambda_j\,\, (j=1,2,\dots)$ are the eigenvalues of $\alpha T_{\sigma}$ with multiplicities taken into account.
\end{cor}

\begin{proof} Note that $\|\alpha(\cdot)\sigma(\xi)\|_{\velp}^r =\|\alpha\|_{\velp}^r|\sigma(\xi)|^r$, hence
\[\sum\limits_{\xi\in\zn}\|\alpha(\cdot)\sigma(\xi)\|_{\velp}^r=\|\alpha\|_{\velp}^r\sum\limits_{\xi\in\zn}|\sigma(\xi)|^r<\infty.\]
An application of Theorem \ref{ch2} concludes the proof.
\end{proof}
In particular, let us consider the symbol
$\sigma(\xi)=(1+4\pi^2 |\xi|^2)^{-\frac{\tau}{2}}$ for $\tau>0$.
The corresponding multiplication yields the operator $\alpha T_\sigma f=\alpha(I-\Delta)^{-\frac{\tau}{2}}f$ on $\tn$. 
We observe that $\sum\limits_{\xi\in\zn}(1+4\pi^2 |\xi|^2)^{-\frac{r\tau}{2}}<\infty$ if and only if $r\tau>n$. 
Consequently we obtain:

\begin{cor} 
Let $p(\cdot)\in\mathcal{P}(\tn)$. If $0<r\leq 1$, $\alpha\in\velp$, and $r\tau>n$, then 
 $\alpha T_{\sigma}=\alpha(I-\Delta)^{-\frac{\tau}{2}}$ is $r$-nuclear from $\vel$ to $\velq$ for all 
 $q(\cdot)\in\mathcal{P}(\tn)$. 
If additionally $p^+<\infty$ and $q(\cdot)=p(\cdot)$, then $\alpha (I-\Delta)^{-\frac{\tau}{2}}$ is $r$-nuclear on 
$\vel(\tn)$ and   
\[
\Tr(\alpha (I-\Delta)^{-\frac{\tau}{2}})=\int_{\tn}\alpha(x)dx\cdot\sum\limits_{\xi\in\zn}(1+4\pi^2|\xi|^2)^{-\frac{\tau}{2}}.
\]
If additionally $r\leq \frac 23$, then   
\[
\Tr(\alpha (I-\Delta)^{-\frac{\tau}{2}})=\int_{\tn}\alpha(x)dx\cdot\sum\limits_{\xi\in\zn}(1+4\pi^2|\xi|^2)^{-\frac{\tau}{2}}
=\sum\limits_{j=1}^{\infty}\lambda_j,
\]
where $\lambda_j\,\, (j=1,2,\dots)$ are the eigenvalues of $\alpha (I-\Delta)^{-\frac{\tau}{2}}$ 
on $\vel(\tn)$ with multiplicities taken into account.
\end{cor}
\noindent{\bf{Acknowledgements}}\\

\noindent The authors express their gratitude  to the anonymous referee who has pointed out several corrections and valuable suggestions leading to 
 a substantial improvement of the manuscript.



\end{document}